\documentclass[11pt]{amsart}
\usepackage{amsmath,amssymb,  graphics, setspace, graphicx, listings, filecontents, amsfonts, enumitem, amssymb, graphics, setspace, graphicx, listings,xcolor}
\usepackage{epstopdf}
\usepackage{listings}
\usepackage[all]{xy}
\usepackage{verbatim}
\usepackage[margin=1.2in]{geometry}
\theoremstyle{plain}
\newtheorem{propn}{Proposition}[section]
\newtheorem{thm}[propn]{Theorem}
\newtheorem{lemma}[propn]{Lemma}

\theoremstyle{definition}
\newtheorem{defn}[propn]{Definition}

\theoremstyle{remark}
\newtheorem*{rem}{Remark}

\numberwithin{equation}{section}

\begin{document}

\title{Numerically destabilizing minimal discs}
\author{Nicholas Brubaker}
\email{nbrubaker@fullerton.edu}

\author{Thomas Murphy}
\email{tmurphy@fullerton.edu}

\address{(BRUBAKER-MURPHY) Department of Mathematics, California State University Fullerton, 800 N. State College Blvd, Fullerton CA 92831, U.S.A.}

\author{K. Oskar Negron}

\email{koskar@csu.fullerton.edu}

\address{(NEGRON) Department of Physics, California State University Fullerton, 800 N. State College Blvd, Fullerton CA 92831, U.S.A.}

\begin{abstract} When calculating the index of a minimal surface, the set of  smooth functions on a domain  with compact support is the standard setting to describe admissible variations.  We show that the set of admissible variations can be widened in a geometrically meaningful manner by considering the difference of area functional, leading to a more general notion of index.  This  allows us to produce explicit examples of destabilizing perturbations for the fundamental Scherk surface and dihedral Enneper surfaces. In the case of dihedral Enneper surfaces we show that both the classical and our modified  index can be explicitly determined.
\end{abstract}

\maketitle

\section{ Introduction}

\footnotetext{ \textit{The authors thank the Dept. of Mathematics at Cal. State Fullerton for encouraging undergraduate student research, as well as Davi Maximo and David Wiygul for their interest and helpful comments.}}A program of central importance in submanifold geometry is to understand and classify  the submanifolds whose principal curvatures  satisfy a natural algebraic condition. A classical topic in this vein is the study of minimal surfaces in Euclidean space. In this paper we make a contribution to the study of the stability of such surfaces by explicitly identifying destablilizing perturbations for some well-known minimal discs. Our code can easily be adapted to any  given minimal disc of interest.  In the special situation of the Enneper surface of dihedral type we produce a numerical method which  completely determines the index. 
 
 \subsection{Background Material}
 
 Let $\Sigma \subset \mathbb{R}^3$ be an immersed surface in Euclidean three-dimensional space, and denote the induced metric by  $g$.  $\Sigma$ is said to be \emph{minimal} if its mean curvature  vector  $H$ vanishes. Such surfaces arise as critical points of the area functional, in a manner we now describe. 
 
 Throughout this paper $\Sigma$  will    be described with one coordinate patch with parameter domain taken to be a topological disc $\Omega\subset \mathbb{R}^2$:
 $$ \Sigma = \lbrace  \sigma(u,v): (u,v)\in \Omega.\rbrace$$

 From elementary calculus one can compute the normal vector field $\textbf{N}$ to $\Sigma$. This is assumed to exist globally (yielding a  \emph{two-sided} minimal disc) and it is further assumed that  $\textbf{N}$ extends smoothly over the boundary $\partial \Omega$.  A \emph{normal variation of} $\Sigma$ is then given by
 $$
  \sigma_t(u,v)   = \sigma(u,v) + t\varphi(u,v) \textbf{N}
 $$
 where $\varphi \in C^{\infty}(\Omega)$ is a bounded smooth function.  Usually in the literature it is assumed that $\varphi$  lies in the Hilbert space $H^1_0(\Omega)$. Geometrically, the corresponding variation will perturb the surface whilst holding the image of the parametrization fixed outside of a compact subset of $\Omega$.

 Given the induced metric  $g_t$ each surface $ \sigma_t(u,v) $ has area
$$
 A[\varphi, t] = \ \int_{\Sigma} \ dV_g = \ \int_{\Omega} \sqrt{ |{g}_t| }\ du \ dv.
 $$
 Here $|\cdot| : = \text{det} (\cdot)$. It is a standard computation that
 $$
 \frac{d}{dt} \bigg|_{t = 0} A[\varphi, t] = \ \ \int_{\Sigma} 2H \varphi \  dV_g.
 $$
 For this to vanish for all  admissible  variations, we need $H=0$. Thus minimal surfaces arise as critical points of the area functional acting on the space of admissible perturbations.  

 From this perspective it is natural to study the second variation of the area functional at a minimal surface. The second variation of $\Sigma$ corresponding to the compact perturbation  $\varphi$ is calculated as

\begin{align*}
 \frac{d^2}{dt^2} \bigg|_{t = 0} A[\varphi, t] =&  \int_{\Sigma} \| \nabla \varphi\|^2 + 2 \kappa\varphi^2 \  dV_g
  \\ = & \int_{\Omega}   \left(g^{ij}\varphi_i\varphi_j  + 2\kappa\varphi^2 \right) \ \sqrt{|g|} \ du \ dv
 \end{align*}
 where $\kappa$ is the Gaussian curvature.  This is a standard calculation in the subject: we refer the reader to the approach taken in  \cite{b}, \cite{kc2}, \cite{nitsche} as it will be relevant later.

 \begin{defn} A minimal surface $(\Sigma, g)$ is unstable if there exists an compact  perturbation $\varphi$ whose corresponding second variation is negative. The index $\mathrm{ind}_c(\Sigma)$  is the dimension of the subspace of the Hilbert space $\mathcal{H}^1_0$  on which the second variation is negative. 
   \end{defn}
 
Equivalently $\mathrm{ind}_c(\Sigma)$ is the index of the associated differential operator $L = -\Delta + 2\kappa$ acting on $\mathcal{H}^1_0$. If $\mathrm{ind}_c(\Sigma) =0$ the surface is said to be stable.  Note that the fact $\varphi$ has compact support has been utilized here to integrate by parts. Via the operator $L$, the Jacobi equation naturally appears, which is an important tool when explicitly calculating the index for a given minimal surface. We refer the reader to \cite{ft} for further details of this approach.

\section{Statement of Main results}

 \subsection{A modified index}
 
Let us suppose throughout that $\Omega$ is a bounded domain and that $\textbf{N}_u$, $\textbf{N}_v$, $\sigma_u$ and $\sigma_v$ are bounded on $\Omega$. Subscripts are used to denote partial differentiation, with the  exception of $t$ which will be used as  a subscript to denote the free parameter in the variation. All the explicit examples of minimal discs we consider will satisfy these assumptions. 

\begin{defn}
Denote by $\mathcal{F}(\Omega)$ the set of smooth functions  $\varphi\in \mathcal{H}^1(\Omega)$ such that
 all partial derivatives with order $|\alpha| \leq 2$ are bounded on $\Omega$, where  $\alpha$ is a multi-index. 
 \end{defn}

Clearly   $$\mathcal{H}^1_0(\Omega) \subset \mathcal{F}(\Omega) \subset \mathcal{H}^1(\Omega).$$ We will restrict to this subspace as the corresponding perturbations have the following geometric interpretation:

\begin{lemma} Let $\varphi\in \mathcal{F}(\Omega)$. Then:
\begin{enumerate}

\item  $$\sigma_t(u,v): \Omega\rightarrow \mathbb{R}^3$$  defines a family of immersions for $t$ sufficiently small.

\item
The difference in area
 $$
\left| A[\varphi, t] -  A[\varphi, 0]\right| = \int_{\Omega} \left|\sqrt{|g_t|} - \sqrt{|g|}\right| \ du \ dv $$ is finite   for sufficiently small $t$. Thus
 $\Sigma$  arises as a  critical point of the ``difference of area" functional on $\mathcal{F}$. 
 \item  The corresponding second variation is given by  $$\int_{\Omega} (g^{ij}\varphi_i\varphi_j + 2\kappa\varphi^2) \ \sqrt{|g|}   du\ dv$$ 
 \end{enumerate}
 \end{lemma}
 \begin{proof} 
 Using the MacLaurin series expansion
 $$
 \sqrt{1 + x} = 1 + x - \frac{1}{8}x^2 + O(x^3), 
 $$
we have to show that as $t\rightarrow 0$, 
 $$
 |g_t| = |g|(1+ x)
 $$
 where $x\rightarrow 0$. This follows from our assumptions on $\varphi$ and $\Omega$. 
The rest of the proof is analogous to the calculations in \cite{b} and \cite{nitsche}. 

 \end{proof}

 Hence it is  geometrically meaningful to consider the second variation  and the associated differential operator, whose index is denoted $\mathrm{ind}_b(\Sigma)$. Negative eigenvalues will still correspond to perturbations which decrease the area to second order. We will call such perturbations (and their corresponding normal variations) \emph{admissible}. A technical issue to bear in mind is that one cannot longer integrate by parts to obtain a Laplacian term in the formula for the second variation without also taking terms on the boundary into account. As such, we lose the link with the Jacobi equation and the conformal Euclidean Laplacian. Nevertheless, we will use Mathematica to directly calculate in local coordinates.

 \subsection{The fundamental Scherk surface}
\iftrue
\begin{figure}[ht]
\begin{center}
\advance\leftskip-3cm
\advance\rightskip-3cm
\includegraphics[keepaspectratio=true,scale=0.3]{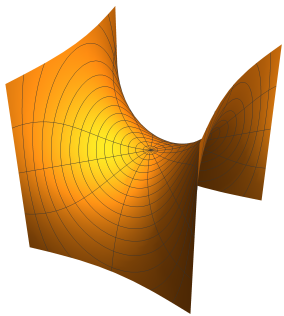}
\caption{The fundamental Scherk surface}
\end{center}\end{figure}
\fi

 Our first result concerns the fundamental (doubly-periodic) Scherk surface. We choose the standard parametrization $ \sigma(u,v) =$
 $$
 \left( \tan^{-1}\left(\frac{2 u}{(u^2 + v^2 - 1}\right), \tan^{-1}\left(\frac{-2v}{u^2 + v^2 - 1}\right), \log\sqrt{\frac{u^2 - v^2 + 1)^2 + 4 u^2 v^2}{(u^2 - v^2 - 1)^2 + 
   4 u^2 v^2}}\right),
 $$
 where $u, v\neq \pm 1$. 
Denote by
 $$
 \Omega_0 = \lbrace (u,v) : -1 < u < 1, \ \ -1 < v< 1 \rbrace
 $$
 the fundamental domain: 
 the surface $\Sigma_0$ given by $\sigma: \Omega_0 \rightarrow \mathbb{R}^3$ is  the \emph{fundamental Scherk surface}. The usual Scherk surface can be viewed as an infinite array of copies of this fundamental building block, one arranged over each square in a checkerboard pattern.

 It is  known \cite{ft} that $ind_{c}(\Sigma_0) = 0$.  Considering therefore the wider class of functions $\mathcal{F}(\Omega_0)$, we have the following: 
 
 \begin{thm}\label{t1} The admissible perturbations $\varphi^i \in\mathcal{F}(\Omega_0)$ associated to the orthogonal test functions 
 \begin{align*}
\varphi^1 &=   (uv)^2+2, \\
\varphi^2 &= u - \frac{1}{4}u^5,\ \ \text{ and}\\
\varphi^3 &= v - \frac{1}{4}v^5
\end{align*}
all destabilize the fundamental Scherk surface.  Moreover,
$\mathrm{ind}_b(\Sigma_0) \geq 3.$

 \end{thm}

\begin{rem} 
Our proof uses Mathematica to calculate explicitly. As the test functions are algebraic the computation will consist of integrating rational functions, and thus we can be confident of the accuracy of the results.  
\end{rem}

\subsection{The dihedral Enneper surfaces $\Sigma_1^n$(R)}

 As an application of our techniques we will consider  the dihedral Enneper surface of order $n$ and radius $R$, labelled $\Sigma_1^n(R)$.  This has $n+1$ ``ripples": our goal is to numerically describe destabilizing perturbations of $\mathrm{ind}_b(\Sigma_1^n)$ which are related to the number of ripples.

 Here the coordinate patch is given by the  parametrization 
$$
\sigma(u,v) = \bigg(\frac{u\beta\cos(v) - u^{\beta}\cos(\beta v)}{2\beta}, \frac{-u\beta\sin(v) - u^{\beta}\sin(\beta v)}{2\beta},
 \frac{u^{\beta}\cos(\beta v)}{\beta}  \bigg)\\
$$
where $\beta = 1 +n$.
\noindent Here  $v\in (0,2\pi)$ and $0<u < R$, with $R$ for now chosen small enough to avoid self-intersections.   To illustrate how  may be used to find destabilizing perturbations, we chose the dihedral Enneper surface with $n=5$. 
\iftrue
\begin{figure}[ht]
\begin{center}
\includegraphics[width=.5\textwidth,trim={0 3.1cm 0 3.1cm},clip]{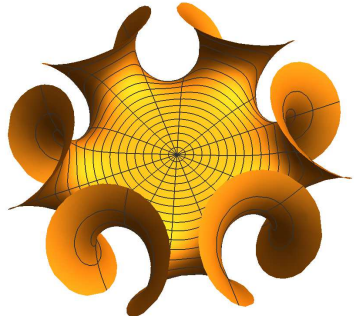}
\caption{An Enneper surface with fifth-order dihedral symmetry.}
\end{center}
\end{figure}
\fi

As a test function, we choose a bump function in the $v$ coordinates restricted to a ``slice" of the surface multiplied by $u$. 
 
 \begin{thm}\label{t2}
For  the surface $\Sigma_1^5(R_0)$,  with $R_0=1.1619...$, let $\chi(v)$ be  the bump function with parameters given by Equation (\ref{bfe}) with $v\in (0, \frac{\pi}{3})$. Then $$\varphi(u,v)  = u\chi(v)$$ destabilizes $\Sigma_1^5(R_0)$. Considering similar test functions where $v$ is successively chosen to range over $(\frac{\pi}{3}, \frac{2\pi}{3}), \dots, (\frac{5\pi}{6}, 2\pi)$, it follows that 
 $$
 \mathrm{ind}_b(\Sigma_1^5(R_0)) \geq 6. 
 $$
 
 \end{thm}
 
 The value of $R_0$ is chosen as it is the maximum radius for which the surface remains embedded. 
 
\begin{rem} In the given Mathematica code, we used dynamic variables for the bump function. As one changes the parameters of the bump function the integral recalculates automatically. This makes it easy to search for a bump function which will work, and easily adapts to any $\Sigma_1^n(R)$. We also encoded the order $n$ of the minimal surface into the calculation, and moreover we restricted the support of the bump function $\chi(v)$ to lie in $(0,\frac{2\pi}{n+1})$. 
\end{rem}
 
 \subsection{Determining the index of $\Sigma_1^n$}
 
 For the Enneper surfaces $\Sigma_1^n(R)$ we  explicitly calculate both $\mathrm{ind}_c(\Sigma^n_1(R))$ and $\mathrm{ind}_b(\Sigma^n_1(R))$  (for any value of $R$) by recasting the problem as a Dirichlet or Neumann boundary value problem for the elliptic differential operator $L$. Our approach is successful for this surface because the patch is in polar coordinates and the singularity of the Laplacian  in these coordinates can be effectively handled via a bordering strategy.
 
 Following the standard conventions (see e.g. \cite{chomax} and \cite{fcs}), define
 \begin{align*}
 \mathrm{ind}_c(\Sigma_1^n) :&= \lim_{R\rightarrow \infty} \mathrm{ind}_c\left(\Sigma_1^n(R)\right)\\
 \mathrm{ind}_b(\Sigma_1^n) :&= \lim_{R\rightarrow \infty} \mathrm{ind}_b\left(\Sigma_1^n(R)\right)\
 \end{align*}
 
 \begin{thm}\label{t3} For the Enneper surface $\Sigma^5_1$, 
 \begin{enumerate}
 \item $\mathrm{ind}_c(\Sigma^5_1) = 9$, and
 \item $\mathrm{ind}_b(\Sigma_1^5) = 11$.
 \end{enumerate}
 \end{thm}

Our code can be adopted to any $\Sigma_1^n$. Numerical experimentation as $n\rightarrow \infty$ suggests the conjecture that
$$
 \mathrm{ind}_c(\Sigma^n_1) = 2n-1, \qquad
 \mathrm{ind}_b(\Sigma_1^n) = 2n+1
$$
One can also determine the corresponding eigenfunctions numerically from our program. 

\subsection{Graphing minimal surfaces} In the appendix we present Mathematica code to graph a minimal surface given the Weierstrass data. Whilst there are several packages available to graph such surfaces given an explicit real parametrization, our approach exploits Mathematica's numerical complex integration function to quickly generate the image. This could be of independent use to readers.

 \section{Proof of Theorem \ref{t1}}
 
\begin{proof} 
Here we give the proof of Theorem \ref{t1}, by providing the Mathematica code for the Scherk surface. The reader can easily adapt it for any other surface by modifying the coordinate patch $\sigma$. 
 
\begin{lstlisting}[frame=single]
Clear[gs,ugs, sigma,NS,FI, FII, Shape,eta]
(************************************************************************)
(******************Replace sigma with the surface patch******************)
(*******************of the desired minimal surface***********************)
(****************and lower and upper bounds for u and v******************)
(************************************************************************)

sigma={ArcTan[(2u)/(u^2+v^2-1)],ArcTan[(-2v)/(u^2+v^2-1)],1/2*Log[((u^2-v^2+1)^2+4u^2 v^2)/((u^2-v^2-1)^2+4u^2 v^2)]}; 
u1=-1;
u2=1;
v1=-1;
v2=1;
\end{lstlisting}

\begin{lstlisting}[frame=single]
(************************************************************************)
(****************Modify eta for different test functions*****************)
(******************to destabilize the second variation*******************)
(************************************************************************)
eta=(uv)^2+2; 

(************************************************************************)
(****************Definitions to calculate stability**********************)
(************************************************************************)
(*We define the metric g_{ij}*)
gs[u_,v_]:={{D[sigma,u].D[sigma,u],D[sigma,u].D[sigma,v]},{D[sigma,v].D[sigma,u],D[sigma,v].D[sigma,v]}}
(*We define the metric g^{ij}*)
ugs[u_,v_]:=Inverse[gs[u,v]]
(*This is the Norm of the gradient for a general surface*)
NormRiemannGrad[f_]:=ugs[u,v][[1]][[1]]D[f,u]D[f,u]+ugs[u,v][[1]][[2]]D[f,u]D[f,v]+ugs[u,v][[2]][[1]]D[f,v]D[f,u]+ugs[u,v][[2]][[2]]D[f,v]D[f,v]
(*The following is the normal to the surface*)
NS=(1/Sqrt[Cross[D[sigma,u],D[sigma,v]].Cross[D[sigma,u],D[sigma,v]]])Cross[D[sigma,u],D[sigma,v]]//Simplify;
(*First fundamental form*)
FI={{D[sigma,u].D[sigma,u],D[sigma,u].D[sigma,v]},{D[sigma,v].D[sigma,u],D[sigma,v].D[sigma,v]}}//Simplify;
(*Second fundamental form*)
FII={{D[sigma,{u,2}].NS,D[sigma,u,v].NS},{D[sigma,v,u].NS,D[sigma,{v,2}].NS}}//Simplify;
(*Shape operator*)
Shape=1/Det[FI] FII.FI;
(*This is Gauss curvature K=Det[Shape]. Alternatively, we could have defined it as Det[FII]/Det[FI]*)
S=Det[Shape]//Simplify; 
(*Definition we used in the derivation of the first variation*)
W=Sqrt[Det[gs[u,v]]]//Simplify;
G=NormRiemannGrad[eta]//Simplify;
(*Calculating the intergral, finally*)
NIntegrate[(G+2*S*eta^2)W,{u,u1,u2},{v,v1,v2}]
\end{lstlisting}

Plugging in the given test functions, 
$$
\frac{d^2}{dt^2} \bigg|_{t = 0} A[\varphi^i, t] =\begin{cases} 
-57.04 & \text{when} \ \ i = 1, \\
-0.05 & \text{when} \ \ i = 2, 3. 
\end{cases}
$$
We will now show that index is at least three, by considering the subspace of $C^{\infty}_b(\Omega)$ spanned by $\varphi^i, i = 1,2,3$. Setting
$\psi = \sum_k c_k\varphi^k$ for constants $c_k$,
\begin{align*}
\frac{d^2}{dt^2} \bigg|_{t = 0} A[\psi, t] =& \int_{\Omega}g_{ij}\psi_i\psi_j + 2\kappa\|\psi\|^2\sqrt{|g|} \ du \ dv \\
=& \sum_{k=1}^3 c_k^2\left( \frac{d^2}{dt^2} \bigg|_{t = 0} A[\varphi^k, t]\right) + \sum_{k\neq m = 1}^3 c_kc_m \int_{\Omega} g_{km}\varphi^k_i\varphi^m_j\ du \ dv
\end{align*}
The first term on the right-hand side is negative by the above calculations. The terms 
$$
\int_{\Omega} g_{km}\varphi^k_i\varphi^m_j\ du \ dv
$$
all vanish for $k\neq m$ because $\sigma$ is an isothermal coordinate patch. 
\end{proof}

\section{Enneper surfaces of dihedral type}

In this section we report on our study of the Enneper surface of dihedral type with  $n=5$. For other values of $n$ we observed similar behaviour; this will lead us to make some conjectures concerning the index of such surfaces.

 \begin{defn}\textbf{Bump function}. Let $h:\bf{R}\rightarrow\bf{R}$ be given by 
 $$h(x)= 
 \begin{cases}  
 \frac{\exp \left(-\frac{a}{a^2-(x-b)^2}\right)}{e^{-\frac{1}{a}}} & \text{ if } | x-b| <| \sqrt{a^2}|\\
0& \text{ otherwise}
 \end{cases}.
$$
 This function yields a bell-like curve that has a maximum at $x=b$. Moreover, this curve approaches zero as $x$ approaches $a$ and it is zero outside of $a$.
 \end{defn}

We this adapt our code to enable bump function with \emph{dynamic} variables. When the interested reader adjusts the parameters, the code will automatically recalculate the second variation formula. Thus the reader can explore stability questions for any given surface of interest.

\subsection{The Proof of Theorem \ref{t2}}
\begin{proof}
Evaluating the parametrization given with  $n=5$ gives us the surface patch 
\begin{align*}
\sigma(u,v)=\left(\frac{1}{22} u \left(11 \cos (v)-u^{10} \cos (11 v)\right),-\frac{1}{22} u \left(u^{10} \sin (11 v)+11 \sin (v)\right),\frac{1}{6} u^6 \cos (6 v)\right).
\end{align*}
Here we  emphasize that $u$ and $v$ are polar coordinates. Before continuing, we need to find the maximal embedded piece. We incorporate the following Mathematica code, which is taken from \cite{indiana}.

\begin{lstlisting}[frame=single]
umin[n0_]:= u/.FindRoot[(Cos[u]-Cot[u+2 n0 u] Sin[u]-2 n0 Cot[u+2 n0 u] Sin[u]),{u,Pi/(2n0+1)+.1,2Pi/(2n0+1)-.1}]
vmin[n0_]:= (Csc[u0+2 n0 u0] (-Sin[u0]-2 n0 Sin[u0]))^(1/(2n0))/.u0->umin[n0]
\end{lstlisting}
Next, we pick a bump function:
\begin{lstlisting}[frame=single]
(*Bump function shifted*)
h[x_,a_,b_]:=Piecewise[{{1/E^(-(1/a))*Exp[-a/(a^2-(x-b)^2)],Abs[x-b]<Abs[Sqrt[a^2]]}},0];

(*Shifted on a block (-k1xk2)x(-m1xm2)*)
n0=5

HigherOrderEnneper[n_][u_,v_]:={(u (Cos[v]+2 n Cos[v]-u^(2 n) Cos[v+2 n v]))/(2+4 n),-((u (Sin[v]+2 n Sin[v]+u^(2 n) Sin[v+2 n v]))/(2+4 n)),(u^(1+n) Cos[(1+n) v])/(1+n)};

sigma=HigherOrderEnneper[n0][u,v];
eps=.001;
k1=0;
k2=rmin[n0];
m1=0;
m2=2Pi/(n0+1);
(*To shrink or expand the region we need to center it at zero*)
Manipulate[(*These variables will be updated as we change the dynamic slide*)global={a,b,c,d};
Grid[{{Plot3D[h[u,a,b] h[v,c,d],{u,k1,k2},{v,m1,m2},PlotRange->All]}}],{a,0,Max[Abs[k1],Abs[k2]]},{b,k1+eps,k2-eps},{c,m1+eps,m2-eps},{d,m1+eps,m2-eps}]
Dynamic@global
h[#[[2]],#[[1]],#[[2]]] h[#[[4]],#[[3]],#[[4]]]&/@Dynamic@global

 \end{lstlisting}
Selecting the  appropriate bump function, 
\begin{equation}\label{bfe}
(a,\ b,\ c,\ d) = (2.66,\ 6.46,\ 4.24,\ 4.15),
\end{equation} the stability code can then run with the following test function:
\begin{lstlisting}[frame=single]
{a1,b1,c1,d1}=global;
eta=u*v*h[v,c1,d1];
\end{lstlisting}

\noindent With these parameters the total integral is $-3.6505$, correct to five decimal places. By symmetry considerations, this is also true for any function where $v$ is chosen to lie in the domains $(\frac{k\pi}{3}, \frac{(k+1)\pi}{3})$, $k= 1, \dots, 5$. This yields six  functions, all with disjoint supports, which destabilize have negative second variation. In this setting it is immediate than any linear combination of these functions will also have negative second variation, so the index is at least six. \end{proof}

\section{Numerically calculating both indexes of $\Sigma_1^n(R)$}

Until now we have taken an elementary approach to determining lower bounds for the index via a trial-and-error choice of test functions. In this section we explain how more advanced numerical techniques can be utilized to  actually determine $\mathrm{ind}_c(\Sigma_1^n(R))$ and $\mathrm{ind}_b(\Sigma^n_1(R))$. For ease of exposition the case $n=5$ will again be our focus.

Recall that the quadratic form for the second variation of area is given by 
\begin{align*}
     Q(\varphi,\varphi) =& \int_\Omega (g^{ij} \varphi_i \varphi_j + 2 \kappa \varphi) \sqrt{g} \, du \, dv, \\
     =& \int_{\Sigma} \varphi \  L(\varphi) dV_g \ \ + \int_{\partial \Sigma} \varphi \varphi_{\textbf{N}} \ d\sigma\\
\end{align*}
where $d\sigma$ denotes the induced measure on $\partial\Sigma$.
Determining the stability of a minimal surface $\Sigma$ is equivalent to the knowing the sign of the minimum value of the Rayleigh quotient  $\mathcal{R}(\varphi) = Q(\varphi,\varphi)/\|\varphi\|^2$ over the set of given admissible functions. In calculating $\mathrm{ind}_c{(\Sigma)}$, we observed in the preceding section that $\mathcal{H}^1_0(\Omega)$ formed the admissible class of functions. Thus we reduce the problem to the classical eigenvalue problem   
\begin{equation}\label{prob:dirichlet}
 L \varphi = \lambda \varphi \quad \text{in } \Omega; \quad 
 \varphi = 0 \quad \text{on } \partial\Omega
\end{equation}
for $\varphi$ in $C^2(\Omega) \cap C(\bar{\Omega})$. 
Using the fact that $L$ is a compact elliptic differential operator, we have here chosen an eigenbasis $\varphi_i$ of $\mathcal{H}^1_0(\Sigma)$ such that $\varphi_i = 0$ on $\partial \Omega.$
This follows from the Lax-Milgram Theorem:  weak coercivity of the associated bilinear form  is deduced from  the fact $2\kappa$ is globally bounded from below on $\Sigma_1^5$.

Thus the number of negative eigenvalues (counted with multiplicities) gives $\mathrm{ind}_c(\Sigma)$.  Note taking the size of the support to infinity defines the canonical index of the corresponding complete surface~\cite{chomax}.

A second, natural choice for the class of admissible functions is $\mathcal{F}$.  Lightly  modifying the above arguments, one can choose an eigenbasis $\{\varphi_i \}$ of $L$  in $H^1(\Omega)$ with the property that
$
\varphi_{\textbf{N}} = 0$ on $  \partial \Omega$
under the additional assumption that $\partial \Omega$ is sufficiently smooth.  

Thus in this case we reduce the problem to determining the number of negative eigenvalues of the classical Neumann problem 
\begin{equation}\label{prob:neumann}
 L \varphi = \lambda \varphi \quad \text{in } \Omega; \quad 
 \varphi_{\mathbf{N}} = 0 \quad \text{on } \partial\Omega
\end{equation}
for $\varphi \in C^2(\Omega) \cap C^1(\bar{\Omega})$. We will then check that the eigenfunctions produced actually lie in the subspace $\mathcal{F}\subset \mathcal{H}^1(\Omega)$ and thus  determine $\mathrm{ind}_b(\Sigma^5_1(R))$. Note that Equations \eqref{prob:neumann} and \eqref{prob:dirichlet} only differ in boundary conditions and will produce different values for the index of a given minimal surface.

For portions of $\Sigma_1^n(R)$ covered by the mapping $\sigma|_{\Omega(R)}$, where  $\Omega(R) = (0,R) \times [0,2\pi)$, the boundary value problems \eqref{prob:dirichlet} and \eqref{prob:neumann} become
\begin{equation}\label{eq:EvalPDE}
 \begin{gathered}
 - (\varphi_{uu}
 + u^{-1}\varphi_u 
 + u^{-2}\varphi_{vv})
 -\frac{8 n^2 u^{2 n-2}}{(1+u^{2 n})^2}\varphi
 = \frac{\lambda(1+u^{2 n})^2}{4}\varphi
 \quad
 (u,v) \in \Omega(R),  \\
 \varphi(u,0) = \varphi(u,2\pi), \qquad
 \varphi_v(u,0) = \varphi_v(u,2\pi), \quad u\in(0,R),
 \end{gathered}
\end{equation}
with 
\begin{equation}\label{eq:EvalDirBC}
 \varphi(R,v) = 0, \quad v \in [0,2\pi), 
\end{equation}
or 
\begin{equation}\label{eq:EvalNeuBC}
 \varphi_u(R,v) = 0, \quad v \in [0,2\pi), 
\end{equation}

respectively. By using a Fourier expansion in the $v$-coordinate, both problems can be reduced to a countable set of Sturm-Liouville systems for functions in terms of the radial coordinate; however the corresponding solutions cannot be explicitly found, and instead we proceed by computing the eigenvalues and eigenvectors of these problems numerically. 

To discretize the differential operator $L$ in equation \eqref{eq:EvalPDE}, we employ  the pseudospectral method~\cite{trefethen2000spectral}, which is a exponentially accurate for analytic functions. Specifically, we represent the domain $\Omega(R)$ via the tensor-product grid $(u_i,v_j)$ for $i = 1,\ldots, N_{u+1}$ and $j = 1,\ldots, N_v$, where the $u_i$'s are non-uniformly spaced Chebyshev-polynomial extreme points on $[0,R]$ and the $v_i$'s are equi-spaced points on $[0,2\pi)$. The eigenvectors $\varphi$ are then represented by an $(N_{u+1}\times N_v)$ matrix $\Phi$ of values on this tensor-product grid, and differentiation by $u$ and $v$ can be approximated by the appropriate right and left matrix multiplications of the standard Chebyshev (for $u$) and Fourier (for $v$) differentiation matrices. Upon vectorizing $\Phi$ via $\mathrm{vec}(\Phi)$, which is achieved by stacking the columns of $\Phi$, the operators on the left and right-hand side of the differential equation in~\eqref{eq:EvalPDE} can be represented by the matrices $A$ and $B$ respectively. Hence $L \varphi = \lambda \varphi$ becomes the linear equation $$A \mathrm{vec}(\Phi) = \lambda B \mathrm{vec}(\Phi).$$
In the $v$-direction, the periodic boundary conditions are automatically employed with the Fourier pseudo-spectral discretization, and in the $u$-direction, the Dirichlet and Neumann conditions at $u = R$ are easily enforced with a bordering strategy; however, at $u = 0$ an artificial singularity is caused by the polar grid. Here, as appropriate smoothness of the surface at $u = 0$ is guaranteed for $\Sigma^n_ 1(R)$, l'H\^{o}™pital's rule implies that 
\[
 u^{-1}\varphi_u = \varphi_{uu}, \qquad u^{-2}\varphi_{vv} = {\varphi_{uuvv}}/{2}
\]
as $u \to 0^+$, which upon substituting into \eqref{eq:EvalPDE} leads to the  boundary equation
\[
 - (2\varphi_{uu}
 + {\varphi_{uuvv}}/{2})
 - 8 n^2\varphi
 = \frac{\lambda}{4}\varphi
\]
at $u = 0$. This equation is then used, with $\varphi(0,0) = \varphi(0,v)$ for $v>0$, in a bordering strategy to replace the rows in $A \mathrm{vec}(\Phi) = \lambda B \mathrm{vec}(\Phi)$ corresponding to $u = 0$. Finally, the resulting generalized eigenvalue problem can then be solved with standard methods from numerical analysis. In our work, the algorithms were implemented and run entirely in MATLAB. The MATLAB code for the Dirchlet problem with $n=5$ and $R=3/4$ is as follows:

\small{
\begin{lstlisting}[frame=single]
% -- Distance -- %
n0 = 5; % n defining the dihedral enneper surface
R0 = .75; % radius

% -- Discretization and and differentiation matrices
N = 60; [Du,u] = cheb(N); u = R0*(u+1)/2; Du = (2/R0)*Du;     % u-discretization 

% - v-discretization - %
M = 60; dt = 2*pi/M; v = -pi + dt*(1:M)';
column = [0 .5*(-1).^(1:M-1).*cot((1:M-1)*dt/2)]';
Dv = toeplitz(column,column([1 M:-1:2]));
Dvv = toeplitz([-pi^2/(3*dt^2)-1/6 .5*(-1).^(2:M)./sin(dt*(1:M-1)/2).^2]);

% -- 2D matrix derivatives and identity -- %
Iu = speye(N+1); Iv = speye(M);
II = kron(Iu,Iv); II = sparse(II);          % Identity

% diff matrices
Duuvv = kron(Du^2,Dvv);
Duu = kron(Du^2,Iv); Duv = kron(Du,Dv); Dvv = kron(Iu,Dvv);
Du = kron(Du,Iv); Dv = kron(Iu,Dv);

% -- Tensor product grids -- %
[uu,vv] = meshgrid(u,v); UU = uu; VV = vv;  
uu = uu(:); vv = vv(:);                         % vectorization of grids

% -- Finding the boundary nodes of Domain -- %
BDnA = uu == 0 | uu == R0;      % All bndry point
BDnO = uu == R0;                % Outer boundary bndry pts at r = R0
BDnI1 = uu == 0 & vv == v(end);                 % Interior boundary bndry pts at r = 0 
BDnI2 = uu == 0 & vv ~= v(end);                 % Interior boundary bndry pts at r = 0


% -- Initial input -- %
% discrete surface (x,y,z) with mean curv. lam and volume V0
x =  uu.*((2*n0+1)*cos(vv) - uu.^(2*n0).*cos((2*n0+1)*vv))/(4*n0+2);
y = -uu.*((2*n0+1)*sin(vv) + uu.^(2*n0).*sin((2*n0+1)*vv))/(4*n0+2);
z = uu.^(n0+1).*cos((n0+1)*vv)/(n0+1);

% initial solution for continuation

% -- e-values preamble -- %
EvalNum = 20;                       % Number of e-values to compute

% Coefficients of the first fundamental form

E = (1/4)*(1 + uu.^(2*n0)).^2;
F = 0*uu;
G = (uu.^2).*E;

W = sqrt(E.*G - F.^2); iW = 4./(uu.*(1 + uu.^(2*n0)).^2);
W2 = E.*G - F.^2; iW2 = 1./(E.*G - F.^2);

% Coefficients of the second fundamental form

n1 = (2*uu.^n0.*cos(n0*vv))./(1 + uu.^(2*n0));
n2 = (2*uu.^n0.*sin(n0*vv))./(1 + uu.^(2*n0));
n3 = (-1 + uu.^(2*n0))./(1 + uu.^(2*n0));

e = -n0*uu.^(n0-1).*cos((n0+1).*vv);
f = n0*uu.^n0.*sin((n0+1).*vv);
g = n0*uu.^(n0+1).*cos((n0+1).*vv);

K = -((16*n0^2*uu.^(2*n0-2))./((1 + uu.^(2*n0)).^4));

% --  Tangent -- %
% L = iW.^2.*G.*Duu - 2*iW.^2.*F.*Duv + iW.^2.*E.*Dvv ...
%     + iW.*(Du*(G.*iW) - Dv*(F.*iW)).*Du + iW.*(Dv*(E.*iW) - Du*(F.*iW)).*Dv - 2*K.*II;

L = (4./((1 + uu.^(2*n0)).^2)).*Duu + (4./(uu.*(1 + uu.^(2*n0)).^2)).*Du ...
    + (4./(uu.^2.*(1 + uu.^(2*n0)).^2)).*Dvv - 2*K.*II;

LBndOp = 8*Duu + 2*Duuvv - 2*K.*II;

J = -L; 

J(BDnO,:) = Du(BDnO,:);
J(BDnI1,:) = LBndOp(BDnI1,:);
J(BDnI2,:) = II(BDnI2,:)-II(BDnI1,:);


% precomputation of B for generalized e-value problem Av = mu*Bv;
B = II; 
B(BDnO,:) = 0*II(BDnO,:);
B(BDnI2,:) = 0*II(BDnI2,:);

% --  computing eigenvalues -- %
[EVEC,EVAL] = eigs(J,B,EvalNum,-100);  EVAL = diag(real(EVAL));
[EVAL,iii] = sort(EVAL); EVEC = EVEC(:,iii);

figure(1)
XX = reshape(x,M,N+1); YY = reshape(y,M,N+1); ZZ = reshape(z,M,N+1);
surf([XX(end,:); XX],[YY(end,:); YY], [ZZ(end,:); ZZ])
shading interp,
xlabel x
ylabel y,
zlabel z
daspect([1 1 1])


% -- plotting the eigenvectors of the parameter domain -- %
for i = 1:6
    
    phi = real(EVEC(:,i))/max(abs(real(EVEC(:,i))));
    Phi = reshape(phi,size(UU)); PHI = [Phi(end,:); Phi];
    
%     figure(2)
%     subplot(3,2,i)
%     UUU = [UU(end,:); UU]; VVV = [-VV(end,:); VV];
%     surf(UUU.*cos(VVV),UUU.*sin(VVV),PHI), shading interp
%     xlabel u
%     ylabel v
%     zlabel 'eigenfunction'
%     daspect([1 1 1])
    
    figure(3)
    subplot(3,2,i)
    surf([XX(end,:); XX],[YY(end,:); YY],[ZZ(end,:); ZZ],PHI), shading flat
    daspect([1 1 1])
    xlabel $x$ 'interpreter' 'LaTeX'
    ylabel $y$ 'interpreter' 'LaTeX'
    zlabel $z$ 'interpreter' 'LaTeX'
    daspect([1 1 1])
    title(['$\lamdbda = ' num2str(EVAL(i),'%f $')],'interpreter','LaTeX')
end

eigenvalues = EVAL


function [D,x] = cheb(N)
if N==0, D=0; x=1; return, end
x = cos(pi*(0:N)/N)';
c = [2; ones(N-1,1); 2].*(-1).^(0:N)';
X = repmat(x,1,N+1);
dX = X-X';
D  = (c*(1./c)')./(dX+(eye(N+1)));      % off-diagonal entries
D  = D - diag(sum(D'));                 % diagonal entries

end


\end{lstlisting}
}

All the code for the subsequent calculations will be omitted from this paper, as they are similar to the code just presented. For the reader's convenience they are available on the authors' websites.  
Figure~\ref{fig:diEn_DirEvec} shows the results of the program run for the Dirichlet problem---i.e., Equations \eqref{eq:EvalPDE} with \eqref{eq:EvalDirBC}---for Enneper's surface with $n = 5$ and  radius $R = 3/4$. Density plots, superimposed on the minimal surface, are given for the first 6 eigenmodes, whose eigenvalues are approximately equal to $39.4193$, $100.5757$, $100.5757$, $181.2596$, $181.2596$ and $210.4917$. Note that two eigenvalues have geometric multiplicity 2, and the first eigenvalue is simple and positive; hence, all of the other eigenvalues must be strictly positive, and the index of this portion of the surface is zero. 

\begin{figure}[!htb]
\begin{center}
 \includegraphics[width=1\textwidth]{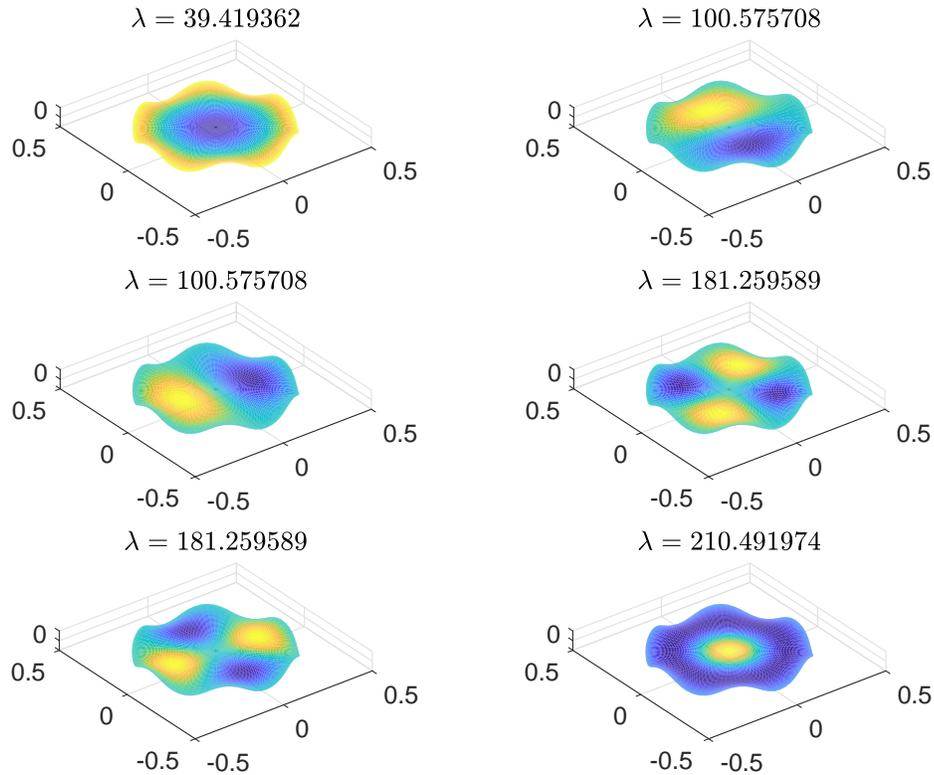}
\caption{Density plot, superimposed on the corresponding Enneper surface, of the first six eigenmodes of \eqref{eq:EvalPDE} with \eqref{eq:EvalDirBC} for $n = 5$ and $R = 3/4$. The eigenvalues $\lambda$ of each mode are additionally displayed.}
\label{fig:diEn_DirEvec}
\end{center}
\end{figure}

As $R$ increased, the index of the surface changes, and these changes are tracked in Figure~\ref{fig:diEn_DirIndex}. The left panel plots the first 11 eigenvalues as a function of the radius $R$ for $R$ between $3/4$ and $R = 1.16169\ldots$, which is the maximal value for which the surface remains embedded. The right panel displays $\mathrm{ind}_c(\Sigma_1^5(R))$, the index of the surface for compact perturbations, for increasing $R$, which means that region of the support of the perturbations is growing. Observe that the minimal eigenvalue has geometric multiplicity $1$ since the index jumps from $0$ to $1$ at $R = 1$. 

Note that care must be taken in increasing the radius $R$ in the code, and to get same order of accuracy seen for small values of $R$, the number of discretization points must be increased as $R$ increases. A good rule of thumb is to make sure there are at least 20 discretization points per unit interval. For the Enneper surface large values of $R$ are not necessary, and computing up to $R = 1.4$ is sufficient for determining its index since the solutions to the Jacobi equation can be computed exactly. This is done by setting $\lambda = 0$, solving the resulting partial differential equation using separation of variables, and finding the roots of a finite degree polynomial in $R$, which depends on the given boundary conditions. These roots determine where the eigenvalues change sign; see~\cite{ft} for a similar computation.

\begin{figure}[t!]
\begin{center}
 \includegraphics[width=1\textwidth,trim={0 3.2cm 0 3.2cm},clip]{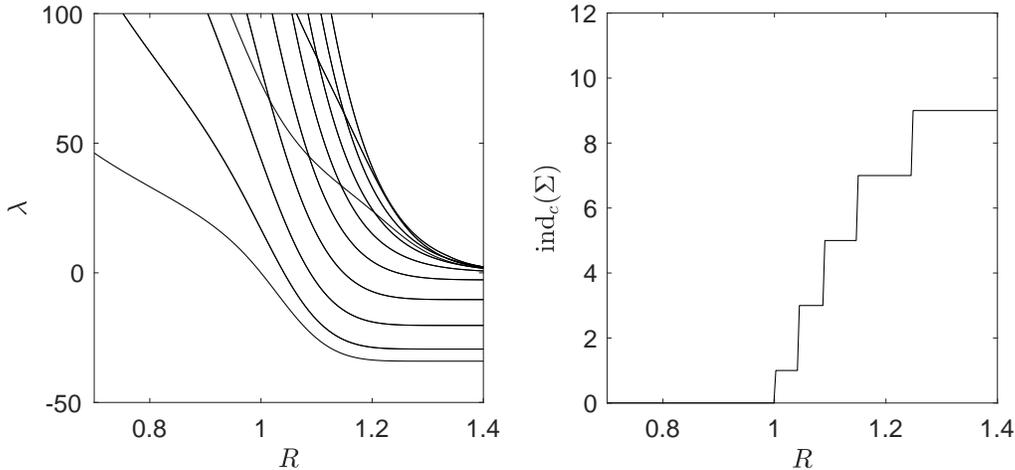}
\caption{Plot, for varying radius $R$, of the eleven smallest eigenvalues of problem \eqref{eq:EvalPDE} with \eqref{eq:EvalDirBC} for $n = 5$ and the corresponding index $\mathrm{ind}_c(\Sigma_1^5)$.}
\label{fig:diEn_DirIndex}
\end{center}
\end{figure}

Figure~\ref{fig:diEn_NeuEvec} shows the results of the program run for the Neumann problem---i.e., Equations \eqref{eq:EvalPDE} with \eqref{eq:EvalNeuBC}---for the same $n$ and $R$ as Figure~\ref{fig:diEn_DirEvec}. Density plots are again displayed for the first 6 eigenmodes, however the corresponding eigenvalues, as expected, are different. Starkly different is the fact that the minimal eigenvalue is already negative. Moreover, the minimum eigenvalue is negative all values of $r$, so the index of the surface, defined by the Neumann problem, is always positive.

How the index changes with varying $R$ for problem \eqref{eq:EvalPDE} with \eqref{eq:EvalNeuBC} is shown Figure~\ref{fig:diEn_NeuIndex}, where the left and right panels, respectively, plot the first 11 eigenvalues and index $\mathrm{ind}_b(\Sigma_1^5)$ as functions of $R$. Again, these plots show only up to the value of $r$ where the surface remains embedded, i.e., $R = 1.16169\ldots$, but this region of $R$ is sufficient to capture the behavior of the index since no more eigenvalue crossings occur for larger values $R$. In turn, $\mathrm{ind}_b(\Sigma_1^5)$ stays equal to $11$ as $R \to \infty$. 

In comparing the results displayed in Figure~\ref{fig:diEn_NeuIndex} to those in Figure~\ref{fig:diEn_DirIndex}, it is apparent that the related indices are indeed different since $\mathrm{ind}_b(\Sigma_1^5)$ is strictly greater than  $\mathrm{ind}_c(\Sigma_1^5)$ for all $R$. In fact, although less widely used, in this situation the Neumann index seems to more closely mimic the theory of complete minimal surface, and in particular Berstein's theorem, since portions of $\Sigma_1^5$ are always unstable for the Neumann index but stable for values of $R < 1$ when considering the Dirichlet index.

\begin{figure}[!htb]
\begin{center}
 \includegraphics[width=1\textwidth]{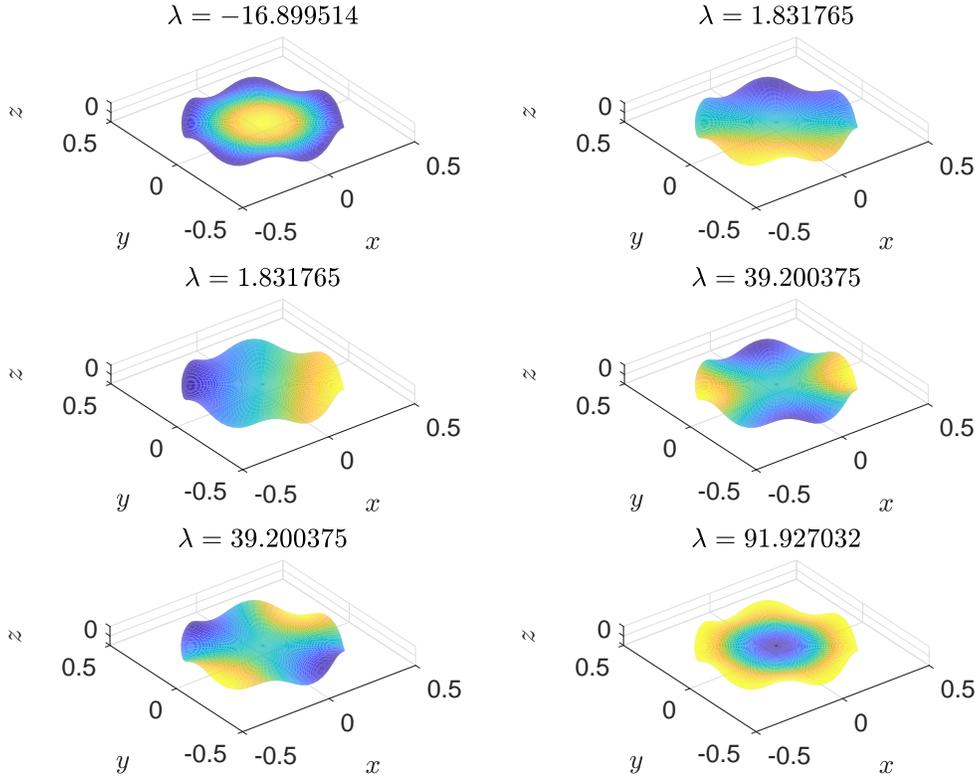}
\caption{Density plot, superimposed on the corresponding Enneper surface, of the first six eigenmodes of \eqref{eq:EvalPDE} with \eqref{eq:EvalNeuBC} for $n = 5$ and $R = 3/4$. The eigenvalues $\lambda$ are also displayed.}
\label{fig:diEn_NeuEvec}
\end{center}
\end{figure}

\begin{figure}[!htb]
\begin{center}
 \includegraphics[width=.8\textwidth,trim={0 3.2cm 0 3.2cm},clip]{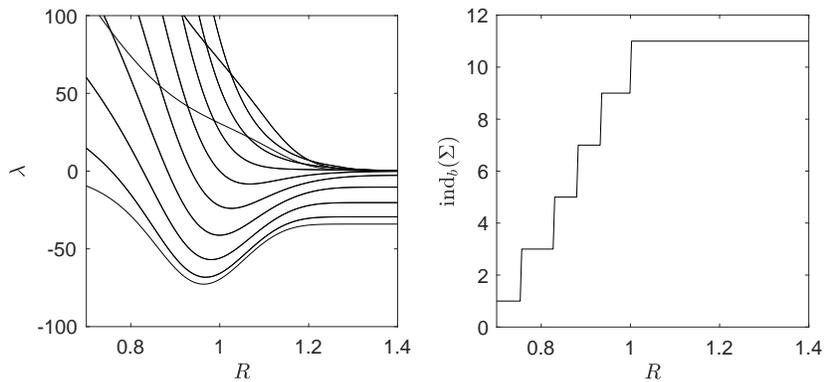}
\caption{Plot, for varying radius $R$, of the eleven smallest eigenvalues of problem \eqref{eq:EvalPDE} with \eqref{eq:EvalNeuBC} for $n = 5$ and  the corresponding index $\mathrm{ind}_b(\Sigma_1^5)$.}
\label{fig:diEn_NeuIndex}
\end{center}
\end{figure}

To emphasize the differences between the indices, observe that $\mathrm{ind}_c(\Sigma_1^{n}(R))$ is a non-decreasing function of $R$ because one can trivially extend any eigenfunction on a ball of fixed radius to any ball with larger radius. This is not necessarily true for $\mathrm{ind}_c(\Sigma_1^{n}(R))$. For example, running the code with $n=11$ and varying $R$ yields the plot shown in Figure~\ref{fig:diEn_NeuIndexN11}.

\begin{figure}[!htb]
\begin{center}
 \includegraphics[width=.8\textwidth,trim={0 3.2cm 0 3.2cm},clip]{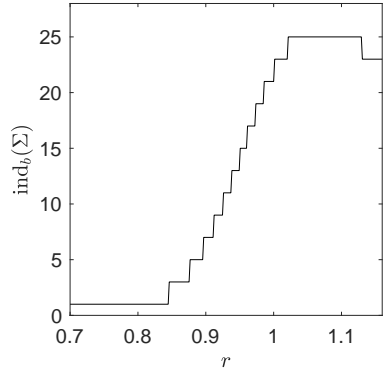}
\caption{The index $\mathrm{ind}_b(\Sigma_1^{11}(R))$ for varying radius $R$, which is not monotonic.}
\label{fig:diEn_NeuIndexN11}
\end{center}
\end{figure}

\vspace{2in}

\section*{Appendix}
\subsection*{ Graphing from the Enneper-Weirstrauss Data}
The following is a code to genereate the image of the  minimal surface with the Enneper-Weirstrauss formalism. We found it to be easy to utilize to quickly graph a minimal surface.  
\begin{lstlisting}[frame=single]
Clear[f, g, x1, y1, z1];
(*Modify for different minimal surface Weierstrauss data*)
f[w_] = 1;
g[w_] = Sqrt[w];

x1[w_] = Integrate[f[w] (1 - g[w]^2)/2, w];
x1[w_] = x1[w] - x1[0];
y1[w_] = Integrate[I*f[w] (1 + g[w]^2)/2, w];
z1[w_] = Integrate[f[w] g[w], w];
ParametricPlot3D[
 Re[{x1[r*Exp[I*\[Theta]]], y1[r*Exp[I*\[Theta]]], 
   z1[r*Exp[I*\[Theta]]]}], {r, 0, 10}, {\[Theta], 0, 2 Pi}, 
 PlotStyle -> Directive[Gray, Specularity[1, 20]], Axes -> None, 
 Mesh -> True, Boxed -> False]
sigma = Re[{x1[r*Exp[I*\[Theta]]], y1[r*Exp[I*\[Theta]]], 
    z1[r*Exp[I*\[Theta]]]}] // Simplify
\end{lstlisting}

\end{document}